\documentclass[12pt]{article}

\usepackage{amsmath, amssymb, amsfonts, amsbsy, amscd, latexsym, amsthm}
\date{}

\newtheorem{Theorem}{Theorem}[section]

\theoremstyle{definition}
\newtheorem{Definition}[Theorem]{Definition}

\theoremstyle{remark}
\newtheorem{Remark}[Theorem]{Remark}

\numberwithin{equation}{section}

\title{On a class of transformations of sequences of complex numbers}

\author{Ilia D. Mishev
\footnote{Department of Mathematics, University of Colorado at Boulder,
Campus Box 395, Boulder, CO 80309-0395, U.S.A.
E-mail address: ilia.mishev@colorado.edu}}

\begin{document}

\maketitle

\begin{abstract}

In this paper we consider a transformation $L_a$ of sequences of complex numbers. We find the inverse transformation of $L_a$ as well as the inverse of a related transformation $\tilde{L}_a$. We explore a connection to the binomial transform and significantly generalize a previously known result. We also obtain new relations among many classical hypergeometric orthogonal polynomials as well as new formulas for sums involving terminating hypergeometric series.

\end{abstract}

\section{Introduction}

Let $\omega$ denote the complex linear space of all sequences
$x=(x_n)_{n=0}^{\infty}$ of complex numbers. Let $a_{n,k}$ be 
complex numbers, where $n$ and $k$ are integers with $n \geq 0, 0 \leq k \leq n$. We
can define a linear transformation
\begin{eqnarray*}
&&L:\omega \to \omega \\
&&x=(x_n)_{n=0}^{\infty} \mapsto L(x)=(L(x)_n)_{n=0}^{\infty}
\end{eqnarray*}
by
\begin{equation}
\label{103} 
L(x)_n=\sum_{k=0}^n a_{n,k} x_k, \quad n \geq 0.
\end{equation}

One classical example of such a transformation is the binomial transform
(introduced by Knuth in \cite{Knuth}) defined by
$$\hat{x}_n=\sum_{k=0}^n (-1)^k \binom{n}{k} x_k,
\quad n \geq 0.$$

Given a transformation $L$ of the kind defined in (\ref{103}),
$L$ is invertible if and only if $a_{n,n} \neq 0$ for all $n \geq 0$.
If $L$ is invertible, its inverse transformation $L^{-1}$ is of the form 
$$L^{-1}(x)_n=\sum_{k=0}^n b_{n,k} x_k, \quad n \geq 0,$$
for some $b_{n,k} \in \mathbb{C}$ with $n \geq 0, 0 \leq k \leq n$. 
The inverse transformation can also be written as
$$x_n=\sum_{k=0}^n b_{n,k} L(x)_k, \quad n \geq 0.$$

As an example, the binomial transform is invertible and is equal to
its own inverse (see \cite{Rior}), which can be expressed by
$$x_n=\sum_{k=0}^n (-1)^k \binom{n}{k} \hat{x}_k,
\quad n \geq 0.$$

\begin{Remark}
\label{R110}
We note that if
$$y_n=\sum_{k=0}^n a_{n,k} x_k \Leftrightarrow
x_n=\sum_{k=0}^n b_{n,k} y_k$$
is a pair of inverse transformations and 
$(c_n)_{n=0}^{\infty}$ is a sequence of nonzero complex numbers, then
$$y_n=\sum_{k=0}^n a_{n,k} c_k x_k \Leftrightarrow
x_n=\frac{1}{c_n}\sum_{k=0}^n b_{n,k} y_k$$
is another pair of inverse transformations. 
\end{Remark}

A large number of pairs of inverse transformations are
examined by Riordan in \cite{Rior}.

In this paper, we will consider the following family of transformations:

\begin{Definition}
\label{D110}
Let $a \in \mathbb{C} \backslash 
\{-1, -2, -3, \ldots\}$. We define the transformation 
\begin{eqnarray*}
&&L_a:\omega \to \omega \\
&&x=(x_n)_{n=0}^{\infty} \mapsto L_a(x)=(L_a(x)_n)_{n=0}^{\infty}
\end{eqnarray*}
by
\begin{equation}
\label{120}
L_a(x)_n=\sum_{k=0}^n (-n)_k (n+a)_k x_k, \quad n \geq 0,
\end{equation}
where,
for $\gamma \in \mathbb{C}$ and $k$ a nonnegative integer,
the rising factorial $(\gamma)_k$ is given by
\begin{equation*}
(\gamma)_k=\left\{ \begin{array}{ll}
\gamma(\gamma+1)\cdots(\gamma+k-1), & k>0,\\
1, & k=0.
\end{array} \right.
\end{equation*}
\end{Definition}

We require that $a \notin \{-1, -2, -3, \ldots\}$ since
otherwise the transformation $L_a$ is not invertible as in that case
$(n+a)_n=0$ for $1 \leq n \leq |a|$.

The transformation $L_a$ arises naturally from the definitions of
many classical hypergeometric orthogonal polynomials. It also comes up in the summation formulas
of many terminating hypergeometric series.

\begin{Remark}
\label{R120}
From the inverse transformation of the binomial transform, Remark \ref{R110}, and the formula
$$(-1)^k\binom{n}{k}
=\frac{(-n)_k}{k!},$$
we obtain the pair of inverse transformations
$$y_n=\sum_{k=0}^n (-n)_k x_k \Leftrightarrow
x_n=\frac{1}{n!}
\sum_{k=0}^n \frac{(-n)_k}{k!} y_k.$$
The transformation $L_a$ given in Definition \ref{D110} has the form
$$y_n=\sum_{k=0}^n (-n)_k (n+a)_k x_k, \quad
a \in \mathbb{C} \backslash 
\{-1, -2, -3, \ldots\}.$$
A related transformation $\tilde{L}_a$ (see Definition \ref{D310}) that we will also study has the form
$$y_n=\sum_{k=0}^n \frac{(-n)_k}{(1+a+n)_k} x_k, \quad
a \in \mathbb{C} \backslash 
\{-2, -3, -4, \ldots\}.$$
\end{Remark}

In Section 3 we will find the inverse of the $L_a$ transformation using
a classical summation formula due to Dixon. 
We will also find the inverse of the related $\tilde{L}_a$ transformation.

The connection to the binomial transform 
of a slightly modified version
$L_{a,b}$ (see Definition \ref{D410}) of
the $L_a$ transformation 
is studied in Section 4. 
Theorem \ref{T410} gives a result that
significantly generalizes a previously known special case
described in Remark \ref{R420}. 

In Section 5 we will apply the inverse 
transformation of $L_a$ to deduce new relations for many classical 
hypergeometric orthogonal polynomials. 
In particular, we derive relations for the Wilson, Racah, 
continuous Hahn, Hahn, and Jacobi polynomials as well as for the special cases of 
the Jacobi polynomials given by the Gegenbauer (or ultraspherical) polynomials, 
the Chebyshev polynomials of the first and second kind, 
and the Legendre (or spherical) polynomials. We also use Theorem \ref{T410} to demonstrate other relations for some of the orthogonal polynomials.

In Section 6 we will use the inverse transformations of $L_a$ and $\tilde{L}_a$ to derive new formulas for sums that involve terminating ${}_4F_3(1)$ hypergeometric series and sums that involve terminating ${}_5F_4(1)$ hypergeometric series. These new summation formulas are given in (\ref{610}), (\ref{620}), (\ref{630}), and (\ref{640}).

\section{Hypergeometric series}

Let $p$ and $q$ be nonnegative integers. 
Let $a_1, a_2, \ldots, a_p, b_1, b_2, \ldots, b_q, z \in \mathbb{C}$.
The hypergeometric series of type ${}_pF_q$ that has 
numerator parameters $a_1, a_2, \ldots, a_p$ and 
denominator parameters $b_1, b_2, \ldots, b_q$ is defined by
\begin{equation}
\label{210}
 {}_pF_q \left( \left. {\displaystyle a_1,a_2,\ldots,a_p
\atop \displaystyle b_1,b_2,\ldots,b_q} \right|
z\right) =
\sum_{n=0}^{\infty} \frac{(a_1)_n(a_2)_n \cdots
(a_p)_n}{n!(b_1)_n(b_2)_n \cdots (b_q)_n}z^n.
\end{equation}

If no numerator parameter is a negative integer or zero, 
we need no denominator parameter to be a negative integer or zero.
In this case, the series in (\ref{210}) 
converges absolutely for all $z$ if $p<q+1$. If $p>q+1$, the series
converges only when $z=0$. In the case $p=q+1$, 
the series converges absolutely if $|z|<1$ or if
$|z|=1$ and $\textrm{Re}(\sum_{i=1}^qb_i-\sum_{i=1}^{q+1}a_i)>0$ 
(see \cite[p.\ 8]{Ba}).

If a numerator parameter is a negative integer or zero, then, 
letting $-n$ be the nonpositive integer
numerator parameter of least absolute value, 
only the first $n+1$ terms of the series (\ref{210}) 
are nonzero and the series
is said to terminate. In this case, we require that no 
denominator parameter be in the set
$\{-n+1, -n+2, \ldots, 0 \}$. We note that (\ref{210}) 
reduces to a polynomial in $z$ of degree $n$. 

When $z=1$, we say that the series is of unit argument and of type
${}_pF_q(1)$. If $\sum_{i=1}^qb_i-\sum_{i=1}^pa_i=1$, the
series is called Saalsch\"utzian. In the case $p=q+1$, if
$1+a_1=b_1+a_2=\cdots=b_q+a_{q+1}$, the series is called 
well-poised. A well-poised series that satisfies
$a_2=1+\frac{1}{2}a_1$ is called very-well-poised.

In deriving the inverse transformation of $L_a$, we will make use
of a version of Dixon's theorem (see \cite[p.\ 13]{Ba}) that gives
the sum of a terminating well-poised ${}_3F_2(1)$ series:
\begin{equation}
\label{220}
{}_3F_2 \left( \left. {\displaystyle a,b,-n
\atop \displaystyle 1+a-b,1+a+n}\right| 
1\right) 
=\frac{(1+a)_n\left(1+\frac{a}{2}-b\right)_n}
{\left(1+\frac{a}{2}\right)_n(1+a-b)_n}. 
\end{equation}
 
We will also use the Chu-Vandermonde formula 
(see \cite[p.\ 3]{Ba})
for the sum of a terminating ${}_2F_1(1)$ series:
\begin{equation}
\label{230}
{}_2F_1 \left( \left.{\displaystyle -n,a
\atop \displaystyle b} \right|
1\right) 
=\frac{(b-a)_n}{(b)_n}.
\end{equation}

A terminating hypergeometric series with a nonpositive integer
numerator parameter of least absolute value equal to $-n$ can be considered as a transformation
of the form (\ref{103}). As an example, certain terminating
hypergeometric series can be considered as binomial transforms. Indeed,
Eq.\ (\ref{230})
can be written as
\begin{equation*}
\sum_{k=0}^n (-1)^k \binom{n}{k} \frac{(a)_k}{(b)_k}
=\frac{(b-a)_n}{(b)_n},
\end{equation*}
and therefore we can conclude that the binomial transform of the sequence
$\left(\frac{(a)_n}{(b)_n}\right)_{n=0}^{\infty}$ is the sequence
$\left(\frac{(b-a)_n}{(b)_n}\right)_{n=0}^{\infty}$.

\section{The inverse of the $L_a$ transformation}

In this section we first prove a formula for the inverse of the $L_a$ transformation.
The inverse formula will be used in Sections 5 and 6 to obtain relations
among classical orthogonal orthogonal polynomials as well as formulas for sums that involve terminating 
hypergeometric series.

\begin{Theorem}
\label{T310}
Let $x=(x_n)_{n=0}^{\infty}$ be a sequence of complex numbers.
\begin{enumerate}
\item[(a)] For $a \in \mathbb{C} \backslash
\{0, -1, -2, -3, \ldots\}$,
\begin{equation}
\label{310}
L_a^{-1}(x)_n=\frac{1}{n!(1+a)_n}
\sum_{k=0}^n \frac{(a)_k \left(1+\frac{a}{2}\right)_k (-n)_k}
{k! \left(\frac{a}{2}\right)_k (1+a+n)_k} x_k, \quad n \geq 0. 
\end{equation}
\item[(b)] When $a=0$,
\begin{equation}
\label{320}
L_0^{-1}(x)_n=\frac{1}{(n!)^2}
\left( -x_0 + 2\sum_{k=0}^n \frac{(-n)_k}{(1+n)_k} x_k \right), 
\quad n \geq 0. 
\end{equation}
\end{enumerate}
 
\end{Theorem}

\begin{proof}

\begin{enumerate}

\item[(a)]
When $a \in \mathbb{C} \backslash \{0, -1, -2, -3, \ldots\}$,
we have
\begin{eqnarray*}
&&\frac{1}{n!(1+a)_n}
\sum_{k=0}^n \frac{(a)_k \left(1+\frac{a}{2}\right)_k (-n)_k}
{k! \left(\frac{a}{2}\right)_k (1+a+n)_k} L_a(x)_k \\
&&=\frac{1}{n!(1+a)_n}
\sum_{k=0}^n \left( \frac{(a)_k \left(1+\frac{a}{2}\right)_k (-n)_k}
{k! \left(\frac{a}{2}\right)_k (1+a+n)_k} 
\sum_{m=0}^k (-k)_m (k+a)_m x_k \right) \\
&&=\frac{1}{n!(1+a)_n}
\sum_{m=0}^n\sum_{k=m}^n
\frac{(-1)^m (a)_{k+m} \left(1+\frac{a}{2}\right)_k (-n)_k}
{(k-m)! \left(\frac{a}{2}\right)_k (1+a+n)_k} x_m \\
&&=\frac{1}{n!(1+a)_n}
\sum_{m=0}^n\sum_{t=0}^{n-m}
\frac{(-1)^m (a)_{2m+t} \left(1+\frac{a}{2}\right)_{m+t} (-n)_{m+t}}
{t! \left(\frac{a}{2}\right)_{m+t} (1+a+n)_{m+t}} x_m \\
&&=\frac{1}{n!(1+a)_n}
\sum_{m=0}^n \left(
\frac{(-1)^m (a)_{2m} \left(1+\frac{a}{2}\right)_m (-n)_m}
{\left(\frac{a}{2}\right)_m (1+a+n)_m} \right. \\
&&\left. \times
{}_3F_2 \left( \left. {\displaystyle a+2m,1+\frac{a}{2}+m,-(n-m)
\atop \displaystyle \frac{a}{2}+m,1+a+n+m} \right|
1\right) 
x_m \right).
\end{eqnarray*}
By Eq.\ (\ref{220}),
\begin{equation*}
{}_3F_2 \left( \left.{\displaystyle a+2m,1+\frac{a}{2}+m,-(n-m)
\atop \displaystyle \frac{a}{2}+m,1+a+n+m} \right| 
1\right) 
=\left\{ \begin{array}{ll}
0, & 0 \leq m < n,\\
1, & m=n.
\end{array} \right.
\end{equation*}
Hence
\begin{eqnarray*}
&&\frac{1}{n!(1+a)_n}
\sum_{k=0}^n \frac{(a)_k \left(1+\frac{a}{2}\right)_k (-n)_k}
{k! \left(\frac{a}{2}\right)_k (1+a+n)_k} L_a(x)_k \\
&&=\frac{(-1)^n (a)_{2n} \left(1+\frac{a}{2}\right)_n (-n)_n}
{n!(1+a)_n\left(\frac{a}{2}\right)_n(1+a+n)_n} x_n \\
&&=\frac{(a)_{2n}\left(1+\frac{a}{2}\right)_n}
{(1+a)_{2n}\left(\frac{a}{2}\right)_n} x_n \\
&&=x_n,
\end{eqnarray*}
which proves part (a).

\item[(b)]
Let $a=0$. Taking the limit as $a \to 0$ on the
right-hand side of (\ref{310}), we obtain,
for $n \geq 0$,
\begin{eqnarray*}
&&L_0^{-1}(x)_n
=\lim_{a \to 0}
\frac{1}{n!(1+a)_n}
\sum_{k=0}^n \frac{(a)_k \left(1+\frac{a}{2}\right)_k (-n)_k}
{k! \left(\frac{a}{2}\right)_k (1+a+n)_k} x_k\\
&&=\lim_{a \to 0}
\frac{1}{n!(1+a)_n}
\left( x_0 +
\sum_{k=1}^n \frac{(a)_k \left(1+\frac{a}{2}\right)_k (-n)_k}
{k! \left(\frac{a}{2}\right)_k (1+a+n)_k} x_k \right)\\
&&=\lim_{a \to 0}
\frac{1}{n!(1+a)_n}
\left( x_0 +
2\sum_{k=1}^n \frac{(a+1)_{k-1} \left(1+\frac{a}{2}\right)_k (-n)_k}
{k! \left(1+\frac{a}{2}\right)_{k-1} (1+a+n)_k} x_k \right)\\
&&=\frac{1}{(n!)^2}
\left( x_0 +
2\sum_{k=1}^n \frac{(k-1)! k! (-n)_k}
{k! (k-1)! (1+n)_k} x_k \right)\\
&&=\frac{1}{(n!)^2}
\left( x_0 +
2\sum_{k=1}^n \frac{(-n)_k}
{(1+n)_k} x_k \right)\\
&&=\frac{1}{(n!)^2}
\left( -x_0 +
2\sum_{k=0}^n \frac{(-n)_k}
{(1+n)_k} x_k \right).
\end{eqnarray*}

\end{enumerate}
 
\end{proof}

Considering Theorem \ref{T310}, we define the following related transformation to $L_a$:

\begin{Definition}
\label{D310}
Let $a \in \mathbb{C} \backslash
\{-2, -3, -4, \ldots\}$. We define the transformation 
\begin{eqnarray*}
&&\tilde{L}_a:\omega \to \omega \\
&&x=(x_n)_{n=0}^{\infty} \mapsto 
\tilde{L}_a(x)=(\tilde{L}_a(x)_n)_{n=0}^{\infty}
\end{eqnarray*}
by
\begin{equation}
\label{330}
\tilde{L}_a(x)_n=\sum_{k=0}^n 
\frac{(-n)_k}{(1+a+n)_k} x_k, \quad n \geq 0,
\end{equation}
\end{Definition}

We need the restriction 
$a \notin 
\{-2, -3, -4, \ldots\}$ so that the
$\tilde{L}_a$ transformation is well-defined.
The inverse of the $\tilde{L}_a$ transformation is
given in the next theorem:

\begin{Theorem}
\label{T320}
Let $x=(x_n)_{n=0}^{\infty}$ be a sequence of complex numbers.
\begin{enumerate}
\item[(a)] For $a \in \mathbb{C} \backslash
\{0, -1, -2, -3, \ldots\}$,
\begin{equation}
\label{340}
\tilde{L}_a^{-1}(x)_n=\frac{(a)_n \left(1+\frac{a}{2}\right)_n}
{n!\left(\frac{a}{2}\right)_n}
\sum_{k=0}^n \frac{(-n)_k (n+a)_k}
{k! (1+a)_k} x_k, \quad n \geq 0. 
\end{equation}
\item[(b)] When $a=0$,
\begin{equation}
\label{350}
\tilde{L}_0^{-1}(x)_n=\left\{
\begin{array}{ll}
x_0, & n=0,\\
2\sum_{k=0}^n
\frac{(-n)_k(n)_k}{k!k!}x_k,
& n>0.
\end{array}
\right.
\end{equation}
\item[(c)] When $a=-1$,
\begin{equation}
\label{360}
\tilde{L}_{-1}^{-1}(x)_n=\left\{
\begin{array}{ll}
x_0, & n=0,\\
(2n-1)\sum_{k=0}^n
\frac{(-1)^k(k)_{n-1}}{(n-k)!k!}x_k,
& n>0.
\end{array}
\right.
\end{equation}
\end{enumerate}
 
\end{Theorem}

\begin{proof}
Part (a) follows from (\ref{310}) and Remark \ref{R110}. 
Part (b) follows by taking the limit as $a \to 0$
in part (a) in a similar way as was done in the proof of part (b)
of Theorem \ref{310}. 

To prove part (c), let $a=-1$. 
From the definition of $\tilde{L}_{-1}(x)_n$, it is
clear that $\tilde{L}_{-1}(x)_0=x_0$. Now assume $n>0$.
Taking the limit as $a \to -1$ on the
right-hand side of (\ref{340}), we obtain,
for $n > 0$,
\begin{eqnarray*}
&&\tilde{L}_{-1}^{-1}(x)_n
=\lim_{a \to -1}
\frac{(a)_n \left(1+\frac{a}{2}\right)_n}
{n!\left(\frac{a}{2}\right)_n}
\sum_{k=0}^n \frac{(-n)_k (n+a)_k}
{k! (1+a)_k} x_k\\
&&=\lim_{a \to -1}
\frac{2(1+a)_{n-1} \left(1+\frac{a}{2}\right)_n}
{n!\left(1+\frac{a}{2}\right)_{n-1}}
\left(
\sum_{k=0}^{n-1} \frac{(-n)_k (n+a)_k}
{k! (1+a)_k} x_k 
+\frac{(-n)_n (n+a)_n}
{n! (1+a)_n} x_n
\right)\\
&&=\lim_{a \to -1}
\frac{2(1+a)_{n-1} \left(n+\frac{a}{2}\right)}
{n!}
\left(
\sum_{k=0}^{n-1} \frac{(-n)(n+a)_k}
{k! (1+a)_k} x_k 
+\frac{(-n)_n (n+a)_n}
{n! (1+a)_n} x_n
\right)\\
&&=\lim_{a \to -1}
\frac{2n+a}
{n!}
\left(
\sum_{k=0}^{n-1} \frac{(-n)_k (n+a)_k (1+a+k)_{n-1-k}}
{k!} x_k \right.\\
&&\left. +\frac{(-n)_n (n+a)_n}
{n! (n+a)} x_n
\right)\\
&&=\lim_{a \to -1}
\frac{2n+a}
{n!}
\left(
\sum_{k=0}^{n-1} \frac{(-n)_k (n+a)_k (1+a+k)_{n-1-k}}
{k!} x_k \right.\\
&&\left. +\frac{(-n)_n (n+a+1)_{n-1}}
{n!} x_n
\right)\\
&&=\frac{2n-1}
{n!}
\left(
\sum_{k=0}^{n-1} \frac{(-n)_k (n-1)_k (k)_{n-1-k}}
{k!} x_k 
+\frac{(-n)_n (n)_{n-1}}
{n!} x_n
\right)\\
&&=\frac{2n-1}
{n!}
\left(
\sum_{k=0}^{n-1} \frac{(-n)_k (k)_{n-1}}
{k!} x_k 
+\frac{(-n)_n (n)_{n-1}}
{n!} x_n
\right)\\
&&=(2n-1)
\left(
\sum_{k=0}^{n-1} \frac{(-1)^k (k)_{n-1}}
{(n-k)!k!} x_k 
+\frac{(-1)^n (n)_{n-1}}
{n!} x_n
\right)\\
&&=(2n-1)\sum_{k=0}^n
\frac{(-1)^k(k)_{n-1}}{(n-k)!k!}x_k.
\end{eqnarray*}

\end{proof}

We remark that (\ref{360}) can also be written as
\begin{equation}
\label{370}
\tilde{L}_{-1}^{-1}(x)_n=\left\{
\begin{array}{ll}
x_0, & n=0,\\
x_0-x_1, & n=1,\\
(1-2n)\sum_{k=0}^{n-1}
\frac{(1-n)_k(n)_k}{k!(2)_k}x_{k+1},
& n>1.
\end{array}
\right.
\end{equation}
Indeed, when $n=1$,
\begin{equation*}
(2n-1)\sum_{k=0}^n
\frac{(-1)^k(k)_{n-1}}{(n-k)!k!}x_k
=\sum_{k=0}^1 
\frac{(-1)^k}{(1-k)!k!}x_k
=x_0-x_1.
\end{equation*}
When $n>1$,
\begin{eqnarray*}
&&(2n-1)\sum_{k=0}^n
\frac{(-1)^k(k)_{n-1}}{(n-k)!k!}x_k\\
&&=(2n-1)\sum_{k=1}^n
\frac{(-1)^k(n+k-2)!}{(n-k)!(k-1)!k!}x_k\\ 
&&=(2n-1)\sum_{k=1}^n
\frac{(-1)^k(n-k+1)_{2k-2}}{(k-1)!k!}x_k\\
&&=(2n-1)\sum_{k=1}^n
\frac{(-1)^k(n-k+1)_{k-1}(n)_{k-1}}{(k-1)!k!}x_k\\
&&=(2n-1)\sum_{k=1}^n
\frac{(-1)^k(-1)^{k-1}(1-n)_{k-1}(n)_{k-1}}{(k-1)!k!}x_k\\
&&=(1-2n)\sum_{k=0}^{n-1}
\frac{(1-n)_k(n)_k}{k!(k+1)!}x_{k+1}\\
&&=(1-2n)\sum_{k=0}^{n-1}
\frac{(1-n)_k(n)_k}{k!(2)_k}x_{k+1}.
\end{eqnarray*}

\section{Connection to the binomial transform}

In this section we explore a connection with the binomial transform.
For this purpose, we need to modify slightly the definition of
the $L_a$ transformation and make the following definition:

\begin{Definition}
\label{D410}
Let $a,b \in \mathbb{C}, a \notin 
\{-1, -2, -3, \ldots\}, b \notin 
\{0, -1, -2, -3, \ldots\}$. We define the transformation 
\begin{eqnarray*}
&&L_{a,b}:\omega \to \omega \\
&&x=(x_n)_{n=0}^{\infty} \mapsto L_{a,b}(x)=(L_{a,b}(x)_n)_{n=0}^{\infty}
\end{eqnarray*}
by
\begin{equation}
\label{410}
L_{a,b}(x)_n=\sum_{k=0}^n 
\frac{(-n)_k (n+a)_k}{k! (b)_k} x_k, \quad n \geq 0.
\end{equation}
\end{Definition}

We note that by Theorem \ref{T310}, it follows that
the inverse of the $L_{a,b}$ transformation is given by:
\begin{enumerate}
\item[(a)] For $a \in \mathbb{C} \backslash
\{0, -1, -2, -3, \ldots\}$,
\begin{equation}
\label{420}
L_{a,b}^{-1}(x)_n=\frac{(b)_n}{(1+a)_n}
\sum_{k=0}^n \frac{(a)_k \left(1+\frac{a}{2}\right)_k (-n)_k}
{k! \left(\frac{a}{2}\right)_k (1+a+n)_k} x_k, \quad n \geq 0. 
\end{equation}
\item[(b)] When $a=0$,
\begin{equation}
\label{430}
L_{0,b}^{-1}(x)_n=\frac{(b)_n}{n!}
\left( -x_0 + 2\sum_{k=0}^n \frac{(-n)_k}{(1+n)_k} x_k \right), 
\quad n \geq 0. 
\end{equation}
\end{enumerate}

A connection between the $L_{a,b}$ transformation and the binomial
transform is given in Theorem \ref{T410} below. The result in Theorem \ref{T410} significantly generalizes a previously known special case described in Remark \ref{R420}. 

\begin{Theorem}
\label{T410}
Let $x=(x_n)_{n=0}^{\infty}$ be a sequence of complex numbers and
let $\hat{x}=(\hat{x}_n)_{n=0}^{\infty}$ be its binomial transform.
Then if $b \notin \{ 1+a, 2+a, 3+a, \ldots \}$, we have
\begin{equation}
\label{440}
L_{a,b}(\hat{x})_n
=(-1)^n\frac{(1+a-b)_n}{(b)_n}
L_{a,1+a-b}(x)_n. 
\end{equation}
\end{Theorem}

\begin{proof}
We have
\begin{eqnarray*}
&&L_{a,b}(\hat{x})_n \\
&&=\sum_{k=0}^n 
\frac{(-n)_k (n+a)_k}{k! (b)_k} \hat{x}_k \\
&&=\sum_{k=0}^n \left(
\frac{(-n)_k (n+a)_k}{k! (b)_k} 
\sum_{m=0}^k (-1)^m \binom{k}{m} x_m \right) \\
&&=\sum_{m=0}^n \sum_{k=m}^n
\frac{(-1)^m(-n)_k (n+a)_k}{(b)_k m! (k-m)!}
x_m \\ 
&&=\sum_{m=0}^n \sum_{t=0}^{n-m}
\frac{(-1)^m(-n)_{m+t} (n+a)_{m+t}}
{(b)_{m+t} m! t!}
x_m \\ 
&&=\sum_{m=0}^n \left(
\frac{(-1)^m(-n)_m (n+a)_m}
{(b)_m m!} \right. \\
&&\left. \times
{}_2F_1 \left( \left.{\displaystyle -(n-m),n+a+m
\atop \displaystyle b+m} \right| 
1\right)
x_m \right). 
\end{eqnarray*}
By Eq.\ (\ref{230}),
\begin{eqnarray*}
&&{}_2F_1 \left( \left.{\displaystyle -(n-m),n+a+m
\atop \displaystyle b+m} \right| 
1\right) \\
&&=\frac{(b-a-n)_{n-m}}{(b+m)_{n-m}} \\ 
&&=\frac{(b-a-n)_{n-m}}{(-1)^{n-m}(1-b-n)_{n-m}} \\
&&=(-1)^{n+m}\frac{(1+a-b)_n(b)_m}
{(1+a-b)_m(b)_n}. 
\end{eqnarray*}
Therefore,
\begin{eqnarray*}
&&L_{a,b}(\hat{x})_n \\
&&=(-1)^n\frac{(1+a-b)_n}{(b)_n} 
\sum_{m=0}^n
\frac{(-n)_m (n+a)_m}
{m! (1+a-b)_m} x_m \\
&&=(-1)^n\frac{(1+a-b)_n}{(b)_n}
L_{a,1+a-b}(x)_n.
\end{eqnarray*}
\end{proof}

We note that (\ref{440}) can also
be written as
\begin{equation}
\label{445}
\sum_{k=0}^n
\frac{(-n)_k(n+a)_k}{k!(b)_k}
\hat{x}_k
=(-1)^n
\frac{(1+a-b)_n}{(b)_n}
\sum_{k=0}^n
\frac{(-n)_k(n+a)_k}{k!(1+a-b)_k}
x_k. 
\end{equation}

\begin{Remark}
\label{R410}
Some authors define the binomial transform
of a sequence 
$(x_n)_{n=0}^{\infty}$
by
\begin{equation}
\label{447}
y_n=\sum_{k=0}^n \binom{n}{k} x_k,
\quad n \geq 0.
\end{equation}
With the definition from above, Equation (\ref{445})
becomes
\begin{equation}
\label{448}
\sum_{k=0}^n
\frac{(-n)_k(n+a)_k}{k!(b)_k}
y_k
=(-1)^n
\frac{(1+a-b)_n}{(b)_n}
\sum_{k=0}^n (-1)^k
\frac{(-n)_k(n+a)_k}{k!(1+a-b)_k}
x_k. 
\end{equation}
\end{Remark}

It is interesting to note that the special case
$b=\frac{1+a}{2}$ in (\ref{440}) yields
\begin{equation}
\label{450}
L_{a,\frac{1+a}{2}}(\hat{x})_n
=(-1)^n L_{a,\frac{1+a}{2}}(x)_n,  
\end{equation}
i.e.
\begin{equation}
\label{460}
\sum_{k=0}^n
\frac{(-n)_k(n+a)_k}
{k!\left(\frac{1+a}{2}\right)_k}
\hat{x}_k
=(-1)^n
\sum_{k=0}^n
\frac{(-n)_k(n+a)_k}
{k!\left(\frac{1+a}{2}\right)_k}
x_k. 
\end{equation}

When $a$ is a positive odd integer in (\ref{450}),
i.e. when $a=2r+1, r \in \{0,1,2,\ldots\}$,
using 
$$\frac{(n+2r+1)_k}{(r+1)_k}
=\binom{n+k+2r}{k+r}\frac{r!}{(n+r+1)_r},$$
we obtain from (\ref{460}) that
\begin{equation}
\label{470}
\sum_{k=0}^n (-1)^k\binom{n}{k}
\binom{n+k+2r}{k+r}\hat{x}_k
=(-1)^n\sum_{k=0}^n (-1)^k\binom{n}{k}
\binom{n+k+2r}{k+r}x_k.
\end{equation}

\begin{Remark}
\label{R420}
Recently, Sun \cite[Theorem 2.1]{SunZW} proved the
special case of (\ref{470}) when $r=0$ and
the binomial transform is defined as in (\ref{447}).
Using generalized Seidel matrices, Chen \cite[Theorem 3.2]{ChenKW}
proved a slightly different form of (\ref{470}), again with
the binomial transform as defined in (\ref{447}). If we
take $m=n=s$ in \cite[Theorem 3.2]{ChenKW}, we obtain the
special case of (\ref{470}) when $r=0$ that is also proven in 
\cite[Theorem 2.1]{SunZW}.
\end{Remark}

\section{Connection to classical hypergeometric orthogonal polynomials}

In this section we show the connection of the inverse of the $L_a$
transformation to classical hypergeometric orthogonal polynomials.
In particular, we use Theorem \ref{T310} to obtain new relations 
for the Wilson, Racah, continuous Hahn, Hahn, and Jacobi polynomials.
We also give the corresponding relations for the special cases of the 
Jacobi polynomials given by the Gegenbauer (or ultraspherical) polynomials, 
the Chebyshev polynomials of the first and second kind, and 
the Legendre (or spherical) polynomials. 
Using Theorem \ref{T410}, we demonstrate further relations for some of 
the orthogonal polynomials. For each type of orthogonal 
polynomials, we start with its definition 
as given in \cite{KS} and then state our relations. 

\subsection{Wilson polynomials}

The Wilson polynomials $W_n(x^2;a,b,c,d)$ are defined by
\begin{eqnarray*}
&&W_n(x^2;a,b,c,d)\\
&&=(a+b)_n(a+c)_n(a+d)_n\\
&&\times
{}_4F_3 \left( \left.{\displaystyle -n,n+a+b+c+d-1,a+ix,a-ix
\atop \displaystyle a+b,a+c,a+d} \right| 
1\right). 
\end{eqnarray*}

By (\ref{310}), if 
$a+b+c+d \notin \{1,0,-1,-2,-3,\dots\}$,
\begin{eqnarray}
\label{510}
&&\sum_{k=0}^n \left( \frac{
(a+b+c+d-1)_k 
\left(\frac{a+b+c+d+1}{2}\right)_k (-n)_k }
{k!  
\left(\frac{a+b+c+d-1}{2}\right)_k (a+b+c+d+n)_k
(a+b)_k(a+c)_k(a+d)_k} \right.\nonumber \\
&&\left. \times W_k(x^2;a,b,c,d) \right) \nonumber \\ 
&&=\frac
{(a+b+c+d)_n(a+ix)_n(a-ix)_n}
{(a+b)_n(a+c)_n(a+d)_n}, 
\quad n \geq 0.
\end{eqnarray}

When $a+b+c+d=1$, using (\ref{320}) and
the fact that
$W_0(x^2;a,b,c,d)=1$, we obtain
\begin{eqnarray}
\label{515}
&&\sum_{k=0}^n \frac{(-n)_k}
{(1+n)_k 
(a+b)_k(a+c)_k(a+d)_k}
W_k(x^2;a,b,c,d) \nonumber \\
&&=\frac
{n!(a+ix)_n(a-ix)_n+(a+b)_n(a+c)_n(a+d)_n}
{2(a+b)_n(a+c)_n(a+d)_n}, 
\quad n \geq 0.
\end{eqnarray}

\subsection{Racah polynomials}

The Racah polynomials are defined by
\begin{eqnarray*}
\label{240}
&&R_n(\lambda(x);\alpha,\beta,\gamma,\delta) \\
&&={}_4F_3 \left( \left.{\displaystyle -n,n+\alpha+\beta+1,-x,x+\gamma+\delta+1
\atop \displaystyle \alpha+1,\beta+\delta+1,\gamma+1} \right| 
1\right), \\
&&n=0,1,2, \ldots, N,  
\end{eqnarray*}
where
\begin{equation*}
\lambda(x)=x(x+\gamma+\delta+1)
\end{equation*}
and
\begin{equation*}
\alpha+1=-N \textrm{ or } \beta+\delta+1=-N \textrm{ or }
\gamma+1=-N, \textrm{ with } N \textrm{ a non-negative integer}.
\end{equation*}

By (\ref{310}), if 
$\alpha + \beta \notin \{-1,-2,-3,\dots\}$,
\begin{eqnarray}
\label{520}
&&\sum_{k=0}^n \frac{ 
(\alpha+\beta+1)_k 
\left(\frac{\alpha+\beta+3}{2}\right)_k (-n)_k}
{k!  
\left(\frac{\alpha+\beta+1}{2}\right)_k (2+\alpha+\beta+n)_k}
R_k(\lambda(x);\alpha,\beta,\gamma,\delta) \\
&&=\frac
{(\alpha+\beta+2)_n(-x)_n(x+\gamma+\delta+1)_n}
{(\alpha+1)_n(\beta+\delta+1)_n(\gamma+1)_n}, 
\quad n \geq 0. \nonumber
\end{eqnarray}

When $\alpha+\beta=-1$, using (\ref{320}) and
the fact that
$R_0(\lambda(x);\alpha,\beta,\gamma,\delta)=1$, we obtain
\begin{eqnarray}
\label{525}
&&\sum_{k=0}^n \frac{(-n)_k}
{(1+n)_k}
R_k(\lambda(x);\alpha,\beta,\gamma,\delta)\\
&&=\frac
{n!(-x)_n(x+\gamma+\delta+1)_n
+(\alpha+1)_n(\beta+\delta+1)_n(\gamma+1)_n}
{2(\alpha+1)_n(\beta+\delta+1)_n(\gamma+1)_n}, 
\quad n \geq 0. \nonumber
\end{eqnarray}

\subsection{Continuous Hahn polynomials}

The continuous Hahn polynomials $p_n(x;a,b,c,d)$ are defined by
\begin{eqnarray*}
&&p_n(x;a,b,c,d)\\
&&=i^n\frac{(a+c)_n(a+d)_n}{n!}\\
&&\times
{}_3F_2 \left( \left.{\displaystyle -n,n+a+b+c+d-1,a+ix
\atop \displaystyle a+c,a+d} \right| 
1\right). 
\end{eqnarray*}

By (\ref{310}), 
if $a+b+c+d \notin \{1,0,-1,-2,-3,\dots\}$,
\begin{eqnarray}
\label{530}
&&\sum_{k=0}^n \left( (-i)^k \frac{ 
(a+b+c+d-1)_k 
\left(\frac{a+b+c+d+1}{2}\right)_k (-n)_k}
{
\left(\frac{a+b+c+d-1}{2}\right)_k (a+b+c+d+n)_k 
(a+c)_k(a+d)_k} \right. \nonumber \\
&&\left. \times p_k(x;a,b,c,d) \right) \nonumber \\
&&=\frac
{(a+b+c+d)_n(a+ix)_n}
{(a+c)_n(a+d)_n}, 
\quad n \geq 0.
\end{eqnarray}

When $a+b+c+d=1$, using (\ref{320}) and
the fact that
$p_0(x;a,b,c,d)=1$, we obtain
\begin{eqnarray}
\label{535}
&&\sum_{k=0}^n (-i)^k \frac{(-n)_k k!}
{(1+n)_k 
(a+c)_k(a+d)_k}
p_k(x;a,b,c,d) \nonumber \\
&&=\frac
{n!(a+ix)_n+(a+c)_n(a+d)_n}
{2(a+c)_n(a+d)_n}, 
\quad n \geq 0.
\end{eqnarray}

By (\ref{230}), the binomial transform
of the sequence
$\left(\frac{(d-ix)_n}{(a+d)_n}\right)_{n=0}^{\infty}$
is the sequence
$\left(\frac{(a+ix)_n}{(a+d)_n}\right)_{n=0}^{\infty}$.
Using this along with (\ref{440}), we obtain
\begin{eqnarray}
\label{537}
&&p_n(x;a,b,c,d)\nonumber\\
&&=i^n\frac{(a+c)_n(a+d)_n}{n!}\nonumber\\
&&\times
{}_3F_2 \left( \left.{\displaystyle -n,n+a+b+c+d-1,a+ix
\atop \displaystyle a+c,a+d} \right| 
1\right)\nonumber\\ 
&&=(-1)^ni^n\frac{(a+d)_n(b+d)_n}{n!}\nonumber\\
&&\times
{}_3F_2 \left( \left.{\displaystyle -n,n+a+b+c+d-1,d-ix
\atop \displaystyle b+d,a+d} \right| 
1\right)\nonumber\\
&&=(-1)^n
p_n(-x;d,c,b,a).
\end{eqnarray}
Combining the above with the trivial invariances of 
$p_n(x;a,b,c,d)$ under interchanging $c$ and $d$,
we obtain eight identities of the form
\begin{equation}
\label{538}
p_n(x;a,b,c,d)
=(-1)^{kn}
p_n((-1)^kx;x_1,x_2,x_3,x_4), 
\end{equation}
where $k\in\{0,1\}$ and $(x_1,x_2,x_3,x_4)$ is a
permutation of $(a,b,c,d)$ such that
$x_1,x_2\in\{a,b\}, x_3,x_4\in\{c,d\}$
if $k=0$, and
$x_1,x_2\in\{c,d\}, x_3,x_4\in\{a,b\}$
if $k=1$.

\subsection{Hahn polynomials}

The Hahn polynomials $Q_n(x;\alpha, \beta,N)$
are defined by
\begin{eqnarray*}
&&Q_n(x;\alpha, \beta,N) \\
&&={}_3F_2 \left( \left.{\displaystyle -n,n+\alpha+\beta+1,-x
\atop \displaystyle \alpha+1,-N}\right| 
1\right), \\
&&n=0,1,2, \ldots, N.  
\end{eqnarray*}

By (\ref{310}), if 
$\alpha+\beta \notin \{-1,-2,-3,\dots\}$,
\begin{eqnarray}
\label{540}
&&\sum_{k=0}^n \frac{
(\alpha+\beta+1)_k 
\left(\frac{\alpha+\beta+3}{2}\right)_k (-n)_k}
{k!  
\left(\frac{\alpha+\beta+1}{2}\right)_k (2+\alpha+\beta+n)_k}
Q_k(x;\alpha, \beta,N) \nonumber \\
&&=\frac
{(\alpha+\beta+2)_n(-x)_n}
{(\alpha+1)_n(-N)_n}, 
\quad n \geq 0.
\end{eqnarray}

When $\alpha+\beta=-1$, i.e. when $\beta=-\alpha-1$,
using (\ref{320}) and
the fact that
$Q_0(x;\alpha, \beta,N)=1$, we obtain
\begin{eqnarray}
\label{545}
&&\sum_{k=0}^n \frac{(-n)_k}
{(1+n)_k}
Q_k(x;\alpha, -\alpha-1,N) \nonumber \\
&&=\frac
{n!(-x)_n
+(\alpha+1)_n(-N)_n}
{2(\alpha+1)_n(-N)_n}, 
\quad n \geq 0.
\end{eqnarray}

By (\ref{230}), the binomial transform
of the sequence
$\left(\frac{(-N+x)_n}{(-N)_n}\right)_{n=0}^{\infty}$
is the sequence
$\left(\frac{(-x)_n}{(-N)_n}\right)_{n=0}^{\infty}$.
Using this along with (\ref{440}), we obtain
\begin{eqnarray}
\label{547}
&&Q_n(x;\alpha, \beta,N) \nonumber\\
&&={}_3F_2 \left( \left.{\displaystyle -n,n+\alpha+\beta+1,-x
\atop \displaystyle \alpha+1,-N}\right| 
1\right) \nonumber\\
&&=(-1)^n\frac{(\beta+1)_n}{(\alpha+1)_n}
{}_3F_2 \left( \left.{\displaystyle -n,n+\alpha+\beta+1,-N+x
\atop \displaystyle \beta+1,-N}\right| 
1\right) \nonumber\\
&&=(-1)^n\frac{(\beta+1)_n}{(\alpha+1)_n}
Q_n(N-x;\beta, \alpha,N).
\end{eqnarray}

\subsection{Jacobi polynomials}

The Jacobi polynomials $P_n^{(\alpha,\beta)}(x)$ are defined by
\begin{eqnarray*}
&&P_n^{(\alpha,\beta)}(x)\\
&&=\frac{(\alpha+1)_n}{n!}
{}_2F_1 \left( \left.{\displaystyle -n,n+\alpha+\beta+1
\atop \displaystyle \alpha+1}\right| 
\frac{1-x}{2}\right). 
\end{eqnarray*}

By (\ref{310}), if 
$\alpha+\beta \notin \{-1,-2,-3,\dots\}$,
\begin{eqnarray}
\label{550}
&&\sum_{k=0}^n \frac{
(\alpha+\beta+1)_k 
\left(\frac{\alpha+\beta+3}{2}\right)_k (-n)_k}
{ 
\left(\frac{\alpha+\beta+1}{2}\right)_k (2+\alpha+\beta+n)_k
(\alpha+1)_k}
P_k^{(\alpha,\beta)}(x) \nonumber \\
&&=\frac
{(\alpha+\beta+2)_n}
{(\alpha+1)_n}
\left(\frac{1-x}{2}\right)^n, 
\quad n \geq 0.
\end{eqnarray}

When $\alpha+\beta=-1$, i.e. when
$\beta=-\alpha-1$, using (\ref{320}) and
the fact that
$P_0^{(\alpha,\beta)}(x)=1$, we obtain
\begin{eqnarray}
\label{555}
&&\sum_{k=0}^n \frac{(-n)_k k!}
{(1+n)_k (\alpha+1)_k}
P_k^{(\alpha,-\alpha-1)}(x) \nonumber \\
&&=\frac
{n!
\left(\frac{1-x}{2}\right)^n
+(\alpha+1)_n}
{2(\alpha+1)_n}, 
\quad n \geq 0.
\end{eqnarray}

By the binomial theorem, the binomial transform
of the sequence
$\left(\frac{1+x}{2}\right)_{n=0}^{\infty}$
is the sequence
$\left(\frac{1-x}{2}\right)_{n=0}^{\infty}$.
Using this along with (\ref{440}), we obtain
\begin{eqnarray*}
&&P_n^{(\alpha,\beta)}(x)\\
&&=\frac{(\alpha+1)_n}{n!}
{}_2F_1 \left( \left. {\displaystyle -n,n+\alpha+\beta+1
\atop \displaystyle \alpha+1} \right| 
\frac{1-x}{2}\right)\\
&&=(-1)^n\frac{(\beta+1)_n}{n!}
{}_2F_1 \left( \left. {\displaystyle -n,n+\alpha+\beta+1
\atop \displaystyle \beta+1} \right|  
\frac{1+x}{2}\right)\\ 
&&=(-1)^nP_n^{(\beta,\alpha)}(-x).
\end{eqnarray*}
The above formula for the Jacobi polynomials is well-known 
(see \cite[p.\ 58]{Sz}).

There are also several special cases of the Jacobi polynomials
that we consider below.

\subsubsection{Gegenbauer (or ultraspherical) polynomials}

The Gegenbauer (or ultraspherical) polynomials 
$C_n^{(\lambda)}(x)$ are defined by 
\begin{eqnarray*}
&&C_n^{(\lambda)}(x)
=\frac{(2\lambda)_n}{(\lambda+\frac{1}{2})_n}
P_n^{(\lambda-\frac{1}{2},\lambda-\frac{1}{2})}(x)\\
&&=\frac{(2\lambda)_n}{n!}
{}_2F_1 \left( \left.{\displaystyle -n,n+2\lambda
\atop \displaystyle \lambda+\frac{1}{2}}\right| 
\frac{1-x}{2}\right), 
\quad \lambda \neq 0. 
\end{eqnarray*}

By (\ref{310}), 
\begin{eqnarray}
\label{5610}
&&\sum_{k=0}^n \frac{  
(1+\lambda)_k (-n)_k}
{ 
(\lambda)_k (1+2\lambda+n)_k}
C_k^{(\lambda)}(x) \nonumber \\
&&=\frac
{(1+2\lambda)_n}
{\left(\lambda+\frac{1}{2}\right)_n}
\left(\frac{1-x}{2}\right)^n, 
\quad n \geq 0.
\end{eqnarray}

\subsubsection{Chebyshev polynomials of the first kind}

The Chebyshev polynomials of the first kind
$T_n(x)$ are defined by 
\begin{eqnarray*}
&&T_n(x)
=\frac{P_n^{(-\frac{1}{2},-\frac{1}{2})}(x)}
{P_n^{(-\frac{1}{2},-\frac{1}{2})}(1)}\\
&&={}_2F_1 \left( \left.{\displaystyle -n,n
\atop \displaystyle \frac{1}{2}}\right| 
\frac{1-x}{2}\right). 
\end{eqnarray*}

Using (\ref{320}) and
the fact that
$T_0(x)=1$, we obtain, 
\begin{eqnarray}
\label{5710}
&&\sum_{k=0}^n \frac{(-n)_k}
{(1+n)_k}
T_k(x) \nonumber \\
&&=\frac
{n!
\left(\frac{1-x}{2}\right)^n
+\left(\frac{1}{2}\right)_n}
{2\left(\frac{1}{2}\right)_n},
\quad n \geq 0.
\end{eqnarray}

\subsubsection{Chebyshev polynomials of the second kind}

The Chebyshev polynomials of the second kind
$U_n(x)$ are defined by 
\begin{eqnarray*}
&&U_n(x)
=(n+1)\frac{P_n^{(\frac{1}{2},\frac{1}{2})}(x)}
{P_n^{(\frac{1}{2},\frac{1}{2})}(1)}\\
&&=(n+1){}_2F_1 \left( \left.{\displaystyle -n,n+2
\atop \displaystyle \frac{3}{2}}\right| 
\frac{1-x}{2}\right). 
\end{eqnarray*}

By (\ref{310}), 
\begin{eqnarray}
\label{5810}
&&\sum_{k=0}^n \frac{  
(2)_k (-n)_k}
{k!(2+n)_k}
U_k(x) \nonumber \\
&&=\frac
{(3)_n}
{\left(\frac{3}{2}\right)_n}
\left(\frac{1-x}{2}\right)^n, 
\quad n \geq 0.
\end{eqnarray}

\subsubsection{Legendre (or spherical) polynomials}

The Legendre (or spherical) polynomials 
$P_n(x)$ are defined by 
\begin{eqnarray*}
&&P_n(x)
=P_n^{(0,0)}(x)\\
&&={}_2F_1 \left( \left.{\displaystyle -n,n+1
\atop \displaystyle 1}\right| 
\frac{1-x}{2}\right). 
\end{eqnarray*}

By (\ref{310}), 
\begin{eqnarray}
\label{5910}
&&\sum_{k=0}^n \frac{  
\left(\frac{3}{2}\right)_k (-n)_k}
{ 
\left(\frac{1}{2}\right)_k (2+n)_k}
P_k(x) \nonumber \\
&&=(n+1)
\left(\frac{1-x}{2}\right)^n, 
\quad n \geq 0.
\end{eqnarray}

\section{Connection to sums that involve terminating hypergeometric series}

In this section we explore how Theorems
\ref{T310} and \ref{T320} lead to
new formulas for sums that involve terminating ${}_4F_3(1)$ hypergeometric series and 
sums that involve terminating ${}_5F_4(1)$ hypergeometric series. 
These new summation formulas are given in 
(\ref{610}), (\ref{620}), (\ref{630}), and (\ref{640}) below.

\subsection{Summations involving ${}_4F_3(1)$ series}

From \cite[Eq.\ 4.3.4]{Ba},
\begin{eqnarray*}
&&{}_7F_6 \left( \left. {\displaystyle 
a,1+\frac{a}{2},b,c,d,e,-n
\atop \displaystyle 
\frac{a}{2},1+a-b,1+a-c,1+a-d,1+a-e,1+a+n} 
\right| 1 \right)\\
&&=\frac{(1+a)_n(1+a-d-e)_n}
{(1+a-d)_n(1+a-e)_n}
{}_4F_3 \left( \left. {\displaystyle 
1+a-b-c,d,e,-n
\atop \displaystyle 
1+a-b,1+a-c,d+e-a-n} 
\right| 1 \right).
\end{eqnarray*}
Hence by (\ref{340}),
if $a \notin \{0, -1, -2, -3, \ldots\}$,
\begin{eqnarray}
\label{610}
&&\sum_{k=0}^n \left(
\frac{(-n)_k(n+a)_k(1+a-d-e)_k}
{k!(1+a-d)_k(1+a-e)_k} \right. \nonumber \\
&&\left.\times
{}_4F_3 \left( \left. {\displaystyle 
1+a-b-c,d,e,-k
\atop \displaystyle 
1+a-b,1+a-c,d+e-a-k} 
\right| 1 \right)\right)\\ 
&&=\frac
{(b)_n(c)_n(d)_n(e)_n}
{(1+a-b)_n(1+a-c)_n(1+a-d)_n(1+a-e)_n}.\nonumber
\end{eqnarray}

Using $n+a$ in place of $b$
and $b$ in place of $a$ in
\cite[Eq.\ 7.6.2.15]{PBM},
we have 
\begin{eqnarray*}
&&{}_5F_4 \left( \left. {\displaystyle 
-n,n+a,b,b+\frac{1}{2},c
\atop \displaystyle 
\frac{a}{2},\frac{1+a}{2},d,1+2b+c-d} 
\right| 1 \right)\\
&&=\frac{(a-2b)_n}{(a)_n}
{}_4F_3 \left( \left. {\displaystyle 
-n,2b,d-c,1+2b-d
\atop \displaystyle 
1+2b-a-n,d,1+2b+c-d} 
\right| 1 \right).
\end{eqnarray*}
Hence by (\ref{310}),
if $a \notin \{0, -1, -2, -3, \ldots\}$,
\begin{eqnarray}
\label{620}
&&\sum_{k=0}^n \left(
\frac{(-n)_k\left(1+\frac{a}{2}\right)_k(a-2b)_k}
{k!(1+a+n)_k\left(\frac{a}{2}\right)_k}
\right.\nonumber\\
&&\left.\times
{}_4F_3 \left( \left. {\displaystyle 
-k,2b,d-c,1+2b-d
\atop \displaystyle 
1+2b-a-k,d,1+2b+c-d} 
\right| 1 \right)\right)\\
&&=\frac
{(1+a)_n(b)_n\left(b+\frac{1}{2}\right)_n(c)_n}
{\left(\frac{a}{2}\right)_n\left(\frac{1+a}{2}\right)_n(d)_n(1+2b+c-d)_n}.\nonumber
\end{eqnarray}

\subsection{Summations involving ${}_5F_4(1)$ series}

From \cite[Eq.\ 4.6.1]{Ba},
\begin{eqnarray*}
&&{}_4F_3 \left( \left. {\displaystyle 
b,c,d,-n
\atop \displaystyle 
1+a-c,1+a-d,1+a+n} 
\right| 1 \right)\\
&&=\frac{(1+a)_n(1+a-c-d)_n}
{(1+a-c)_n(1+a-d)_n}\\
&&\times
{}_5F_4 \left( \left. {\displaystyle 
1+\frac{a-b}{2},\frac{1+a-b}{2},c,d,-n
\atop \displaystyle 
1+\frac{a}{2},1+a-b,\frac{1+a}{2},c+d-a-n} 
\right| 1 \right).
\end{eqnarray*}
Hence by (\ref{340}),
if $a \notin \{0, -1, -2, -3, \ldots\}$,
\begin{eqnarray}
\label{630}
&&\sum_{k=0}^n \left(
\frac{(-n)_k(n+a)_k(1+a-c-d)_k}
{k!(1+a-c)_k(1+a-d)_k} \right. \nonumber \\
&&\left.\times
{}_5F_4 \left( \left. {\displaystyle 
1+\frac{a-b}{2},\frac{1+a-b}{2},c,d,-k
\atop \displaystyle 
1+\frac{a}{2},1+a-b,\frac{1+a}{2},c+d-a-k} 
\right| 1 \right) \right) \\
&&=\frac
{(b)_n\left(\frac{a}{2}\right)(c)_n(d)_n}
{(a)_n\left(1+\frac{a}{2}\right)_n(1+a-c)_n(1+a-d)_n}. \nonumber
\end{eqnarray}

From \cite[Eq.\ 4.6.2]{Ba},
\begin{eqnarray*}
&&{}_5F_4 \left( \left. {\displaystyle 
b,1+\frac{a}{2},c,d,-n
\atop \displaystyle 
\frac{a}{2},1+a-c,1+a-d,1+a+n} 
\right| 1 \right)\\
&&=\frac{(1+a)_n(1+a-c-d)_n}
{(1+a-c)_n(1+a-d)_n}\\
&&\times
{}_5F_4 \left( \left. {\displaystyle 
\frac{a-b}{2},\frac{1+a-b}{2},c,d,-n
\atop \displaystyle 
\frac{a}{2},1+a-b,\frac{1+a}{2},c+d-a-n} 
\right| 1 \right).
\end{eqnarray*}
Hence by (\ref{340}),
if $a \notin \{0, -1, -2, -3, \ldots\}$,
\begin{eqnarray}
\label{640}
&&\sum_{k=0}^n \left(
\frac{(-n)_k(n+a)_k(1+a-c-d)_k}
{k!(1+a-c)_k(1+a-d)_k} \right.\nonumber\\
&&\left.\times
{}_5F_4 \left( \left. {\displaystyle 
\frac{a-b}{2},\frac{1+a-b}{2},c,d,-k
\atop \displaystyle 
\frac{a}{2},1+a-b,\frac{1+a}{2},c+d-a-k} 
\right| 1 \right) \right)\\ 
&&=\frac
{(b)_n(c)_n(d)_n}
{(a)_n(1+a-c)_n(1+a-d)_n}.\nonumber
\end{eqnarray}

\end{document}